\documentclass[fleqn,10pt,twocolumn]{SICE_ISCS}
\usepackage{amsmath,amsxtra,amssymb,latexsym, amscd}
\usepackage[mathscr]{eucal}
\def\leq{\leqslant}
\def\geq{\geqslant}

\newtheorem{lemma}[theorem]{Lemma}

\newtheorem{definition}[theorem]{Definition}

\numberwithin{equation}{section}
\def\h2{\hskip 2 cm}

\title{Existence and Stability of Periodic Solutions of a Lotka-Volterra System }
\author{Nguyen Thi Hoai Linh$^{\dagger}$, Ta Hong Quang$^{\ddagger}$, T\d{a} Vi$\hat{\d{e}}$t T$\hat{\rm o}$n$^{\dagger}$}
\affils{$^{\dagger}$Department of Information and Physical Sciences, Graduate School of Information Science and Technology, Osaka University, Osaka, Japan\\
E-mail: taviet.ton(at)ist.osaka-u.ac.jp\\
$^{\ddagger}$Department of Mathematics, Hanoi II Pedagogical  University, Vinh Phuc, Vietnam
}
\abstract{%
In this paper, we study a Lotka-Volterra model which contains  two prey and one predator with the Beddington-DeAngelis functional responses. First, we establish a set of sufficient conditions for existence of positive periodic solutions. Second,  we investigate global asymptotic stability of boundary periodic solutions. Finally, we present some numerical examples.}

\keywords{%
Lotka-Volterra system; Periodic solution; Asymptotic stability; Lyapunov function.
}

\begin{document}

\maketitle

%-----------------------------------------------------------------------

\section{Introduction}
The dynamical relationship between predators and prey has been studied by several authors for a long time. In those researches, various forms of functional responses have been used. Here a functional response means  the average number of prey killed per individual predator per unit of time. Some biologists have argued that in many situation, especially when predators have to search for food, functional responses should depend on both prey's and predator's densities, see \cite{APS,Do,JE,SG} and references therein. 

Let us consider a population of three species, say a Lotka-Volterra model,  with the following properties:
\begin{itemize}
\item [\rm (i)] one species is a predator of two competitive other species.
\item [\rm (ii)]  the predator consumes prey with the functional response given by Beddington \cite{Be} and DeAngelis et al. \cite{DGO}
\end{itemize}
 There are many models having the property (i) or (ii) with diffusion in a constant environment \cite{CC1,CC2,CD}, \cite{Linh}, \cite{TVT1,TVT2,TVT3}. 
 However, natural environments are usually periodic in time due to the periodicity of seasons. Therefore,  the parameters in these models should be periodic in time. This paper devotes to studying such a  
 Lotka-Volterra model which is performed by a nonlinear system of differential equations:
%%%%%%%%%%%%%%%%%%%%%%%%   E1
\begin{eqnarray}
\begin{array}{rcl}
x'_1&=&x_1[a_1(t)-b_{11}(t) x_1-b_{12}(t)x_2]\\
&&- \frac{c_1(t)x_1x_3}{\alpha(t)+\beta(t) x_1+\gamma(t) x_3},\\
x'_2&=&x_2[a_2(t)-b_{21}(t) x_1-b_{22}(t)x_2]\\
&&- \frac{c_2(t)x_2x_3}{\alpha(t)+\beta(t) x_2+\gamma(t) x_3},\\
x'_3&=&x_3\Big[-a_3(t)+ \frac{d_1(t)x_1}{\alpha(t)+\beta(t) x_1+\gamma(t) x_3}\\
&&+\frac{d_2(t)x_2}{\alpha(t)+\beta(t) x_2+\gamma(t) x_3}\Big].
\end{array}
\label{E1}
\end{eqnarray}
Here $x_i(t)$ represents the population density of species $X_i$ at time $t \, (i\geq 1),$ $X_3$ is a predator species and $ X_1, X_2$ are  competitive prey  species.
At time $t$, $a_i(t)$ is the intrinsic growth rate of  $X_i \, (i=1,2)$ and $a_3(t)$ is the death rate of $X_3$; $b_{ij}(t)$  measures the amount of competition between $X_1$ and $X_2 \, (i\ne j, i,j\leq 2),$ and $b_{ii}(t)\, (i\leq 2)$ measures the inhibiting effect of environment on $X_i.$  The predator consumes prey with functional responses: 
$$ \frac{c_1(t)x_1x_3}{\alpha(t)+\beta(t) x_1+\gamma(t) x_3} \text{ and } \frac{c_2(t)x_2x_3}{\alpha(t)+\beta(t) x_2+\gamma(t) x_3};$$
 and contributes to its growth with amounts: 
 $$\frac{d_1(t)x_1}{\alpha(t)+\beta(t) x_1+\gamma(t) x_3} \text{ and } \frac{d_2(t)x_2}{\alpha(t)+\beta(t) x_2+\gamma(t) x_3}.$$
Furthermore, we assume that the parameters $a_i(t), b_{ij}(t), $ $c_i(t),$ $ d_i(t),  \alpha(t), \beta(t), \gamma(t) (1\leq i, j\leq 3) $ are $\omega$-periodic and continuous in $t$ and bounded below by some positive constants.

 In the next section, we present our main results. First, we use the continuation theorem in coincidence degree theory to show existence of  positive periodic solutions of \eqref{E1}. Second, by using   Lyapunov functions we verify global asymptotic stability  of  boundary periodic solutions. Finally, we give numerical examples.
\section{Main results}
For biological reasons we only consider \eqref{E1} with nonnegative initial values, i.e. $x_1(0), x_2(0),x_3(0)\geq 0.$ 
Let $g(t)$ be a function, for a brevity, instead of writing $g(t)$ we write $g$. If $g$ is a bounded continuous function on $\mathbb R$, we denote\\
\centerline{$g^u= \sup_{t \in \mathbb R}\ g(t), \ g^l=\inf_{t \in \mathbb R} g(t),$}
and
$\hat g=\frac{1}{\omega} \int_0^\omega g(t)dt,$
if $g$ is a periodic function with period $\omega$. 
% %%%%%%%%%%%%%%%%%%%%%%%%%%%%%% Dinh Nghia 3
\begin{definition}
A nonnegative solution $x^*(t)$ of \eqref{E1} is called a global asymptotic stable solution  if it attracts any other solution $x(t)$ of \eqref{E1} in the sense that\\
\centerline{$\lim_{t \to \infty}\sum_{i=1}^3|x_i(t)-x_i^*(t)|=0.$}
\end{definition}
\subsection{Existence of  positive periodic solutions}
In this subsection, we shall study existence of  periodic solutions of  \eqref{E1}.  It is not difficult to verify  global existence and uniqueness of nonnegative solutions of  \eqref{E1}.  To show the existence of a positive periodic solution, we shall use the continuation theorem in coincidence degree theory which has been used for some mathematical models of Lotka-Volterra type \cite{L3,TVT2} and references therein. The following are some concepts and results taking from \cite{GM}.

Let $\mathbb X$ and $\mathbb Y$ be two Banach spaces. A linear mapping $L\colon{\mathcal D}(L) \subset \mathbb X \to \mathbb Y$ is called {\it Fredholm} if  it satisfies two conditions:
\begin{itemize}
\item [\rm (i)] $\mathop{\rm Im} L$ is closed and has finite codimension;
\item [\rm (ii)] Ker $L$ has finite dimension.
\end{itemize}
The {\it index} of $L$ is the integer $\mathop{\rm dim} \mathop{\rm Ker} L -\mathop{\rm codim}\mathop{\rm Im} L$.
If $L$ is Fredholm  of index zero, there exist continuous projections $ P\colon  \mathbb X \to \mathbb X$ and $ Q\colon \mathbb Y \to \mathbb Y$ such that $\mathop{\rm Im} P=\mathop{\rm Ker} L, \mathop{\rm Im} L=\mathop{\rm Ker} Q=\mathop{\rm Im} (I-Q),$ and an isomorphism $J\colon  \mathop{\rm Im} Q \to \mathop{\rm Ker} L$. 
It follows that\\
\centerline{$L_p=L|_{{\mathcal D}(L) \cap \mathop{\rm Ker} P}\colon  (I-P) {\mathbb X }\to \mathop{\rm Im} L$}
is invertible. We denote the inverse of that map
by $K_p$. Let $\Omega$ be an open bounded subset of $\mathbb X$. A continuous mapping $N\colon\mathbb X \to\mathbb Y$ is said to be $L$-{\it compact} on $\bar\Omega$ if the following two conditions take place:
\begin{itemize}
\item [\rm (i)] the mapping $QN\colon  \bar\Omega \to \mathbb Y$ is continuous and bounded;
\item [\rm (ii)] $K_p(I-Q)N\colon\bar \Omega\to \mathbb X$ is compact, i.e. it is continuous and   $K_p(I-Q)N(\bar \Omega)$ is relatively compact. 
\end{itemize}
 To introduce the definition of the degree of  $N$ in $\Omega$, for simplicity we assume that ${\mathbb X}={\mathbb R}^N.$  Suppose furthermore that $N$ is smooth on $\bar \Omega$. Let $p\notin \partial \Omega$ be a regular value of $N$, i.e. the equation $N(x)=p$ on $\bar \Omega$ has only a finite number of solutions $x_1,\dots,x_n\in \Omega$ with  nonsingular $DN(x_i)$ for each $i=1,\dots,n$ where $DN(x_i)$ is the Jacobi matrix of $N$ at $x_i$. Then the degree $\deg(N,\Omega,p) $ of $N$ in $\Omega$ at $p$ is defined by the formula 
$$
\deg(N,\Omega,p)=\sum_{i=1}^n \mathop{\rm sgn}\{\det DN(x_i)\}.
$$
\begin{lemma} [Continuation theorem \cite{GM}] \label{lem3}
Let L be a Fredholm mapping of index $0$. Assume that $N\colon  \bar \Omega \to \mathbb Y$ is $L$-compact
on $\bar \Omega$ and satisfies conditions:
\begin{itemize}
\item [\rm (a)] for each $\lambda \in (0, 1)$ every solution of $Lx = \lambda Nx$ is such that $x \notin \partial \Omega$,
\item [\rm (b)] $QN x 	\ne {\bf 0}$ for each $x \in \partial\Omega \cap \mathop{\rm Ker} L$, and
$\deg \{JQN, \Omega \cap \mathop{\rm Ker} L, {\bf 0}\} \ne 0.$
\end{itemize}
Then the operator equation $Lx = Nx$ has at least one solution in ${\mathcal D}(L) \cap $ $\bar\Omega.$
\end{lemma}

We now put
\begin{equation*}
\begin{aligned} 
&L_{i1}=\ln\frac{\hat a_i}{\hat b_{ii}}, H_{i1}=L_{i1}+ 2 \hat a_i \omega \quad (i=1,2),\\
&L_{12}=\ln \Big\{\frac{\hat a_1-\hat b_{12} e^{H_{21}}-\widehat{(\frac{c_1}{\gamma })}}{\hat b_{11}}\Big\}, \\
&H_{12}=L_{12}-2 \hat a_1 \omega,\\
&L_{22}=\ln \Big\{\frac{\hat a_2-\hat b_{21} e^{H_{11}}-\widehat{(\frac{c_2}{\gamma })}}{\hat b_{22}}\Big\}, \\
&H_{22}=L_{22}-2 \hat a_2 \omega,\\
&L_{31}=\ln\Big\{\frac{\hat d_1e^{H_{11}}+\hat d_2e^{H_{21}}-\hat a_3 \alpha^l}{\hat a_3 \gamma ^l}\Big\}, \\
&H_{31}=L_{31}+2 \hat a_3 \omega,\\
&L_{32}=\ln[(\hat d_1-\hat a_3 \beta^u)e^{H_{12}}\\
&\hspace{1.2cm}+(\hat d_2-\hat a_3 \beta^u)e^{H_{22}}-2\hat a_3 \alpha^u]-\ln(2\hat a_3 \gamma ^u),\\
& H_{32}=L_{32}-2 \hat a_3 \omega.
\end{aligned}
\end{equation*}
The convention here is that $\ln x = -\infty$ if $x\leq  0$. Under the  conditions
%%%%%%%%%% Hyp2
\begin{equation} \label{H2}
\begin{cases}
\hat b_{11}\hat b_{22}\ne \hat b_{12}\hat b_{21},\\
\hat a_1-\hat b_{12} e^{H_{21}}-\widehat{(\frac{c_1}{\gamma })}>0,\\
\hat a_2-\hat b_{21} e^{H_{11}}-\widehat{(\frac{c_2}{\gamma })}>0,\\
\hat d_1e^{H_{11}}+\hat d_2e^{H_{21}}-\hat a_3 \alpha^l>0,\\
(\hat d_1-\hat a_3 \beta^u)e^{H_{12}}+(\hat d_2-\hat a_3 \beta^u)e^{H_{22}}>2\hat a_3 \alpha^u,
\end{cases}
\end{equation}
we shall verify existence of an $\omega$-periodic solution of \eqref{E1}.
%%%%%%%%%%%%%%%%%%%%%%%  Dinh ly thm:5
\begin{theorem} \label{thm:5}
Let   \eqref{H2} be satisfied. Then \eqref{E1} has at least one positive $\omega$-periodic solution.
\end{theorem}
\begin{proof}
  By putting 
$x_i(t)=e^{u_i(t)} (i\geq 1)$,  \eqref{E1} becomes
%%%%%%%%%%%%%%%%%%%%%%%%   Cong thuc E14
\begin{equation}\label{E14}
\begin{cases}
u'_1=a_1-b_{11}e^{u_1} -b_{12}e^{u_2}- \frac{c_1e^{u_3}}{\alpha+\beta e^{u_1}+\gamma e^{u_3}},\\
u'_2=a_2-b_{21}e^{u_1} -b_{22}e^{u_2}- \frac{c_2e^{u_3}}{\alpha+\beta e^{u_2}+\gamma e^{u_3}},\\
u'_3=-a_3+ \frac{d_1e^{u_1}}{\alpha+\beta e^{u_1}+\gamma e^{u_3}}+\frac{d_2e^{u_2}}{\alpha+\beta e^{u_2}+\gamma e^{u_3}}.
\end{cases}
\end{equation}
Let 
\begin{equation*}
\begin{aligned} 
\mathbb X&=\mathbb Y\\
&=\{u=(u_1, u_2, u_3)^T\in  C^1(\mathbb R,\mathbb R^{\text{$3$}})\,{\text{such that}}\\
& \hspace{1cm}u_i(s)=u_i(s+\omega) \text{ for } s\in \mathbb R \text{ and   $i$} \geq \text{$1$}\},
\end{aligned}
\end{equation*}
with norm
$$ ||u||=\sum_{i=1}^3 \max_{s\in [0,\omega]} |u_i(s)|, \hspace{1cm} u\in \mathbb X.$$
Then both $\mathbb X$ and $ \mathbb Y$ are  Banach spaces.
Let 
\begin{equation*}
\begin{aligned}
&N
\begin{bmatrix}
u_1\\
u_2\\
u_3
\end{bmatrix} (s)
=
\begin{bmatrix}
N_1(s)\\
N_2(s)\\
N_3(s)
\end{bmatrix}\\
&
=
\begin{bmatrix}
a_1-b_{11}e^{u_1} -b_{12}e^{u_2}- \frac{c_1e^{u_3}}{\alpha+\beta e^{u_1}+\gamma e^{u_3}}\\
a_2-b_{21}e^{u_1} -b_{22}e^{u_2}- \frac{c_2e^{u_3}}{\alpha+\beta e^{u_2}+\gamma e^{u_3}}\\
-a_3+ \frac{d_1e^{u_1}}{\alpha+\beta e^{u_1}+\gamma e^{u_3}}+\frac{d_2e^{u_2}}{\alpha+\beta e^{u_2}+\gamma e^{u_3}}
\end{bmatrix},\\
&L
\begin{bmatrix}
u_1\\
u_2\\
u_3
\end{bmatrix}
=
\begin{bmatrix}
u_1'\\
u_2'\\
u_3'
\end{bmatrix},\\
&P
\begin{bmatrix}
u_1\\
u_2\\
u_3
\end{bmatrix}=
Q
\begin{bmatrix}
u_1\\
u_2\\
u_3
\end{bmatrix}=
\begin{bmatrix}
\frac{1}{\omega} \int_0^\omega u_1(s) ds\\
\frac{1}{\omega} \int_0^\omega u_2(s) ds\\
\frac{1}{\omega} \int_0^\omega u_3(s) ds
\end{bmatrix}.
\end{aligned}
\end{equation*}
Hence, 
$$\mathop{\rm Ker} L= {\mathbb R^{\text 3}}, \mathop{\rm Im} L=\{u\in \mathbb Y\,{|}\, \text{$\int_0^\omega u_i(s)ds=0, \, i\geq1\}$},$$
and $\mathop{\rm dim} \mathop{\rm Ker} L=3=\mathop{\rm codim} \mathop{\rm Im} L.$
Then, it is easy to obtain the following conclusions.
\begin{enumerate}
  \item  $L$ is a Fredholm mapping of index zero, since $\mathop{\rm Im} L$ is closed in $\mathbb Y$. 
  \item $P$ and $ Q$ are continuous projections such that $\mathop{\rm Im} P=\mathop{\rm Ker} L, \mathop{\rm Im} L=\mathop{\rm Ker} Q=\mathop{\rm Im} (I-Q)$. 
  \item The generalized inverse (to $L$) $K_P\colon\mathop{\rm Im} L \to {\mathcal D}(L)\cap \mathop{\rm Ker}  P$ exists and is given by
\begin{equation*}
K_P
\begin{bmatrix}
u_1\\
u_2\\
u_3
\end{bmatrix} (\nu)
=
\begin{bmatrix}
\int_0^\nu u_1(s)ds-\frac{1}{\omega}\int_0^\omega \int_0^\nu u_1(s)dsd\nu\\
\int_0^\nu u_2(s)ds-\frac{1}{\omega}\int_0^\omega \int_0^\nu u_2(s)dsd\nu\\
\int_0^\nu u_3(s)ds-\frac{1}{\omega}\int_0^\omega \int_0^\nu u_3(s)dsd\nu\\
\end{bmatrix}.
\end{equation*}
\item
 $QN$ and $K_P(I-Q)N$ are  continuous. 
\item  $N$ is $L$-compact on $\bar \Omega$ with any open bounded set $\Omega\subset \mathbb X.$
\end{enumerate}

Let us now find an appropriate open, bounded subset $\Omega$ for application of the continuation theorem.  Obviously, from   \eqref{H2},  $-\infty< L_{i2}\leq L_{i1}<\infty \, (i\geq 1).$ Corresponding to the  equation $Lu=\lambda N u, \lambda\in (0,1),$ we have
%%%%%%%%%%%%%%%%%%%%%%%%%%%%%%%% Cong thuc E15
\begin{align}
&u_1'=\lambda\Big[a_1-b_{11}e^{u_1} -b_{12}e^{u_2}- \frac{c_1e^{u_3}}{\alpha+\beta e^{u_1}+\gamma e^{u_3}}\Big], \notag\\
&u_2'=\lambda\Big[a_2-b_{21}e^{u_1} -b_{22}e^{u_2}- \frac{c_2e^{u_3}}{\alpha+\beta e^{u_2}+\gamma e^{u_3}}\Big],   \label{E15}\\
&u_3'=\lambda\Big[-a_3+ \frac{d_1e^{u_1}}{\alpha+\beta e^{u_1}+\gamma e^{u_3}}\notag\\
&\hspace{1.2cm}+\frac{d_2e^{u_2}}{\alpha+\beta e^{u_2}+\gamma e^{u_3}}\Big].\notag
\end{align}
Suppose that $(u_1, u_2, u_3)^T\in \mathbb X$ is an arbitrary solution of \eqref{E15}. Integrating both the hand sides of \eqref{E15} over the interval $[0,\omega]$, we obtain
%%%%%%%%%%%%%%%%%%%%%%%%%%%%%%%% Cong thuc E16
\begin{align}
&\hat a_1 \omega=\int_0^\omega\Big[b_{11}e^{u_1}+b_{12}e^{u_2}+ \frac{c_1e^{u_3}}{\alpha+\beta e^{u_1}+\gamma e^{u_3}}\Big]dt,  \notag\\
&\hat a_2 \omega=\int_0^\omega\Big[b_{21}e^{u_1}+b_{22}e^{u_2}+ \frac{c_2e^{u_3}}{\alpha+\beta e^{u_2}+\gamma e^{u_3}}\Big]dt,  \label{E16}\\
&\hat a_3 \omega=\int_0^\omega\Big[\frac{d_1e^{u_1}}{\alpha+\beta e^{u_1}+\gamma e^{u_3}}+\frac{d_2e^{u_2}}{\alpha+\beta e^{u_2}+\gamma e^{u_3}}\Big]dt. \notag
\end{align}
Combining the first equations of \eqref{E15} and \eqref{E16}, we observe that 
%%%%%%%%%%%%%%%%%%%%%%%%%%% Cong thuc E17
\begin{equation*} \label{E17}
\begin{aligned}
&\int_0^\omega|u_1'|dt \\
&\leq \lambda\Big[\int_0^\omega a_1 dt+\int_0^\omega b_{11}e^{u_1}dt\\
 &\hspace{1cm} +\int_0^\omega b_{12}e^{u_2}dt+ \int_0^\omega \frac{c_1e^{u_3}}{\alpha+\beta e^{u_1}+\gamma e^{u_3}}\Big]dt\\
 &<  2 \hat a_1 \omega.
 \end{aligned}
\end{equation*}
Similarly, we have
$\int_0^\omega |u_2'(t)|dt < 2 \hat a_2 \omega,$ 
and
\begin{align*}
&\int_0^\omega|u_3'(t)|dt \leq   \lambda\Big[\int_0^\omega a_3 dt+ \int_0^\omega \frac{d_1e^{u_1}}{\alpha+\beta e^{u_1}+\gamma e^{u_3}}dt\\
 &+ \int_0^\omega\frac{d_2e^{u_2}}{\alpha+\beta e^{u_2}+\gamma e^{u_3}}\Big]dt< 2 \hat a_3 \omega.
\end{align*}
Since $u \in \mathbb X$, there exist $\xi_i, \eta_i\in [0, \omega] \, (i\geq 1)$ such that
%%%%%%%%%%%%%%%%%%%%%%%%%%%% cong thuc E18
\begin{equation}\label{E18}
u_i(\xi_i)=\min_{t\in [0,\omega]} u_i(t), \quad u_i(\eta_i)=\max_{t\in [0,\omega]} u_i(t).
\end{equation}
From the first equation of \eqref{E16} and \eqref{E18}, we obtain
\begin{equation*}
\begin{aligned}
\hat a_1 \omega &\geq \int_0^\omega b_{11}e^{u_1(\xi_1)}dt+\int_0^\omega b_{12}e^{u_2(\xi_2)}dt\\
&=\hat b_{11} \omega e^{u_1(\xi_1)}+\hat b_{12} \omega e^{u_2(\xi_2)},
\end{aligned}
\end{equation*}
which implies that $u_1(\xi_1)< L_{11}.$ Hence,  for all $ t\geq 0$ 
$$u_1(t)\leq u_1(\xi_1)+\int_0^\omega|u_1'(t)|dt < L_{11}+ 2 \hat a_1 \omega=H_{11}.$$
Similarly, we have $u_2(t) < H_{21}$ for all $t\geq 0$.\\
 On the other hand, from the first equation of \eqref{E16} and \eqref{E18}, 
\begin{equation*}
\begin{aligned}
&\hat a_1 \omega \\
&\leq \int_0^\omega b_{11}e^{u_1(\eta_1)}dt+\int_0^\omega b_{12}e^{u_2(\eta_2)}dt+\int_0^\omega \frac{c_1(t)}{\gamma (t)}dt\\
&=\hat b_{11} \omega e^{u_1(\eta_1)}+\hat b_{12} \omega e^{u_2(\eta_2)}+\widehat{(\frac{c_1}{\gamma })}\omega\\
&\leq \hat b_{11} \omega e^{u_1(\eta_1)}+\hat b_{12} \omega e^{H_{21}}+\widehat{(\frac{c_1}{\gamma })}\omega.
\end{aligned}
\end{equation*}
Hence, 
$$u_1(t)\geq u_1(\eta_1)-\int_0^\omega |u_1'(t)|dt \geq H_{12},  \hspace{1cm} \forall  t\geq 0.$$
Similarly,  $u_2(t) \geq H_{22}$ for all $t\geq 0.$ Therefore, by putting
$B_i=\max \{|H_{i1}|, |H_{i2}|\}$, we conclude that 
$$\max_{t\in [0, \omega]} |u_i(t)| \leq B_i, \hspace{1cm} i=1,2.$$

Let us give estimates for $u_3(t)$. It follows from the third equation of \eqref{E16} and \eqref{E18} that
\begin{equation*}
\begin{aligned}
\hat a_3 \omega &\leq \int_0^\omega\Big[\frac{d_1(t)e^{H_{11}}}{\alpha^l+\gamma ^l e^{u_3(\xi_3)}}+\frac{d_2(t)e^{H_{21}}}{\alpha^l+\gamma ^l e^{u_3(\xi_3)}}\Big]dt\\
&=\frac{[\hat d_1e^{H_{11}}+\hat d_2e^{H_{21}}]\omega}{\alpha^l+\gamma ^l e^{u_3(\xi_3)}}
\end{aligned}
\end{equation*}
and 
\begin{equation*}
\begin{aligned}
\hat a_3 \omega \geq &\int_0^\omega\Big[\frac{d_1(t)e^{H_{12}}}{\alpha^u+\beta^u e^{H_{12}}+ \gamma ^u e^{u_3(\eta_3)}}\\
&+\frac{d_2(t)e^{H_{22}}}{\alpha^u+\beta^u e^{H_{22}}+\gamma ^u e^{u_3(\eta_3)}}\Big]dt\\
= &\frac{\hat d_1e^{H_{12}}\omega}{\alpha^u+\beta^u e^{H_{12}}+ \gamma ^u e^{u_3(\eta_3)}}\\
&+\frac{\hat d_2e^{H_{22}}\omega}{\alpha^u+\beta^u e^{H_{22}}+\gamma ^u e^{u_3(\eta_3)}}\\
\geq &\frac{[\hat d_1e^{H_{12}}+\hat d_2e^{H_{22}}]\omega}{2\alpha^u+\beta^u [e^{H_{12}}+e^{H_{22}}]+2\gamma ^u e^{u_3(\eta_3)}}.
\end{aligned}
\end{equation*}
Hence, $u_3(\xi_3) \leq L_{31}$ and $u_3(\eta_3) \geq L_{32}.$ We then observe that \\
\centerline{$
u_3(t) \leq u_3(\xi_3)+ \int_0^\omega|u_3'(t)| dt \leq H_{31}
$}
and
$$
u_3(t) \geq u_3(\eta_3)- \int_0^\omega|u_3'(t)| dt \geq H_{32}.$$
Therefore, by putting $B_3=\max \{|H_{31}|, |H_{32}|\},$ we get 
$$\max_{t\in [0, \omega]}|u_3| \leq B_3.$$

By the above estimates, for any solution $u\in \mathbb X$ of \eqref{E16} we have  
$||u|| \leq \sum_{i=1}^3 B_i.$
Clearly, $B_i \, (i\geq 1)$ are independent of $\lambda.$ Take $B=\sum_{i=1}^4 B_i$ where $B_4$ is taken sufficiently large such that 
\text{$B_4\geq \sum_{i=1}^3\sum_{j=1}^2 |L_{ij}|.$}
Let $\Omega=\{u\in \mathbb X \,{|}\, \text{$ ||u||<B$}\},$ 
then $\Omega$ satisfies the condition (a) of Lemma \ref{lem3}. 

Let us verify that the condition (b) of Lemma \ref{lem3} is also satisfied. Consider the homotopy
$$H_{\mu}(u)=\mu QN(u)+(1-\mu)G(u), \hspace{1cm} \mu\in[0,1]$$
where $G\colon \mathbb R^{\text 3} \to \mathbb R^{\text 3}$,
$$G(u)=
\begin{bmatrix}
\hat a_1-\hat b_{11} e^{u_1}-\hat b_{12} e^{u_2}\\
\hat a_2-\hat b_{21} e^{u_1}-\hat b_{22} e^{u_2}\\
-\hat a_3+f(u_1,u_3)+g(u_2,u_3)
\end{bmatrix}$$
with
$f(u_1, u_3)=\frac{1}{\omega} \int_0^\omega \frac{d_1e^{u_3}dt}{\alpha+\beta e^{u_1}+\gamma e^{u_3}}$ and 
 $g(u_2,u_3)= \frac{1}{\omega}\int_0^\omega \frac{d_2e^{u_3}dt}{\alpha+\beta e^{u_2}+\gamma e^{u_3}}.$
It is easy to see that
\begin{align*}
&H_{\mu}(u)\\
&=
\begin{bmatrix}
\hat a_1-\hat b_{11} e^{u_1}-\hat b_{12} e^{u_2}-\frac{1}{\omega}\int_0^\omega \frac{\mu c_1e^{u_3} dt}{\alpha+\beta e^{u_1}+\gamma e^{u_3}}\\
\hat a_2-\hat b_{21} e^{u_1}-\hat b_{22} e^{u_2} -\frac{1}{\omega}\int_0^\omega \frac{\mu c_2e^{u_3}dt}{\alpha+\beta e^{u_2}+\gamma e^{u_3}}\\
-\hat a_3+f(u_1,u_3)+g(u_2,u_3)
\end{bmatrix}.
\end{align*}
By carrying out similar arguments as above, we observe  that any solution $u^*$ of the equation $H_{\mu}(u)={\bf 0} \in \mathbb R^{\text 3}$ with $\mu\in [0,1]$ satisfies the estimate
%%%%%%%%%%%%%%% E22
\begin{equation} \label{E22}
L_{i2}\leq u_i^*\leq L_{i1}, \hspace{1cm} i\geq 1.
\end{equation}
Thus, ${\bf 0}\notin H_{\mu}( \partial \Omega \cap \mathop{\rm Ker} L)$ for $\mu\in [0,1]$. Consequently, by taking $\mu=1$,  we conclude that ${\bf 0}\notin QN(\partial\Omega \cap \mathop{\rm Ker} L)$. Note that the isomorphism $J$ can be the identity mapping $I$, since $\mathop{\rm Im} P=\mathop{\rm Ker} L.$ By the invariance property of homotopy, we obtain that
%%%%%%%%%%%%%%%%%%%% E23
\begin{equation*} \label{E23}
\begin{aligned}
&\deg (JQN,  \Omega  \cap  \mathop{\rm Ker} L, {\bf 0})=\deg(QN,  \Omega \cap \mathop{\rm Ker} L, {\bf 0})\\
&=\deg(QN,  \Omega \cap {\mathbb R^{\text 3}}, {\bf 0})=\deg(G,  \text{$\Omega$} \cap \mathbb R^{\text 3}, {\bf 0})\\
&=\mathop{\rm sgn}\det \Lambda\\
&=\mathop{\rm sgn}\left\{ (\hat b_{11}\hat b_{22}-\hat b_{12}\hat b_{21})  \frac{\partial [f(u_1,u_3)+g(u_2,u_3)]}{\partial u_3}\right\}
\end{aligned}
\end{equation*}
where 
\begin{align*}
&\Lambda=
\begin{bmatrix}
-\hat b_{11} e^{u_1}& -\hat b_{12} e^{u_2} &0\\
-\hat b_{21} e^{u_1}& -\hat b_{22} e^{u_2} &0\\
\frac{\partial f(u_1,u_3)}{\partial u_3} & \frac{\partial g(u_2,u_3)}{\partial u_3} & \frac{\partial f(u_1,u_3)}{\partial u_3}+ \frac{\partial g(u_2,u_3)}{\partial u_3}
\end{bmatrix}.
\end{align*}
Since  both functions $f(u_1, u_3)$ and $ g(u_2,u_3)$  increase in $u_3$,  $
\frac{\partial f(u_1,u_3)}{\partial u_3}+ \frac{\partial g(u_2,u_3)}{\partial u_3}>0.
$ Hence, by using the first condition in \eqref{H2}, we conclude that \\
\centerline{$\deg (JQN,  \Omega  \cap  \mathop{\rm Ker} L, {\bf 0})\ne 0.$}

By now we have proved that $\Omega$ verifies all requirements of Lemma \ref{lem3}. Therefore, the equation  $Lu=Nu$ has at least one solution in ${\mathcal D}(L)\cap \bar \Omega,$ i.e. \eqref{E14} has at least one $\omega$-periodic solution $u^*$ in ${\mathcal D}(L)\cap \bar \Omega.$ Set $x_i^*=e^{u_i^*} (i\geq 1),$ then $x^*$ is an $\omega$-periodic solution of \eqref{E1} with strictly positive components. It completes the proof.
\end{proof}
\subsection{Global asymptotic  stability  of  boundary periodic solutions} 
In this subsection, we shall establish a sufficient criteria for   global asymptotic  stability of boundary $\omega$-periodic solutions  of \eqref{E1}.  Consider  the boundary dynamics of \eqref{E1} where $X_3$ is absent, i.e. $x_3(t)=0$ for every $t\geq 0$. We then consider the  periodic competitive model of two prey $X_1, X_2$:
%%%%%%%%%%%%%%%%%%% E25
\begin{equation}\label{E25}
\begin{cases}
x'_1=x_1\left[a_1(t)-b_{11}(t) x_1-b_{12}(t)x_2\right],\\
x'_2=x_2\left[a_2(t)-b_{21}(t) x_1-b_{22}(t)x_2\right].
\end{cases}
\end{equation}
Denote by $\bar X_i(t)$  the unique positive $\omega$-periodic solution  of the logistic equation:\\
\centerline{$X'=X\left[a_i(t)-b_{ii}(t) X\right].$}
Then
$\bar X_i(t)=\frac{e^{\int_0^\omega a_i(s)ds -1}}{\int_t^{t+\omega} b_{ii}(s) e^{-\int_s^t a(\tau)d\tau}ds}.$
Due to \cite{Lb}, if
%%%%%%%%%%%%%%% E26
\begin{equation} \label{E26}
\hat {a_i} > \widehat {b_{ij} \bar X_j}\hspace{1cm}  (i\ne j, i, j=1,2),
\end{equation}
then  \eqref{E25}  has a positive $\omega$-periodic solution $(\bar x_1, \bar x_2)$. Furthermore, 
if 
%%%%%%%%%%%%%%%%%%%%%%%% E27
\begin{equation} \label{E27}
\hat A_{12}<0
\end{equation}
then $(\bar x_1, \bar x_2)$ is globally asymptotically stable, where
$a_{ij}(t)=b_{ij}(t) \bar x_j(t) \,\, ( i\ne j, i,j=1,2)$ and 
$$A_{12}(t)=\max\left\{ \frac{(a_{ij}+a_{ji})^2}{4 a_{ii}}-a_{jj}, i\ne j\right\}.$$
 Our result is as follows.
\begin{theorem} 
If \eqref{E26} and \eqref{E27} hold then $\bar x=(\bar x_1, \bar x_2, 0)$ is a  $\omega$-periodic boundary solution of \eqref{E1}. Furthermore,
\begin{itemize}
\item [{\rm (i)}]
If $b_{ij}<b_{jj} \, (1\leq i\ne j\leq 2)$ and  $ c_1+c_2+d_1+d_2< \beta a_3$ then $\bar x$ is globally asymptotically stable.
\item[{\rm (ii)}]  If $ c_1+c_2+d_1+d_2< \beta a_3$ then $\bar x$ attracts any solution $x$ of \eqref{E1} which satisfies the condition
 $$[x_1(t)-\bar x_1(t)] [x_2(t)-\bar x_2(t)] \geq 0, \hspace{1cm} \forall t\geq 0.$$
\item[{\rm (iii)}] If $ d_1+d_2< \beta a_3$ then $\bar x$ attracts any solution $x$ of \eqref{E1} which satisfies the condition
 $$x_i(t) \geq \bar x_i(t), \hspace{1cm} \forall t\geq 0, i=1, 2.$$
\end{itemize}
\end{theorem}
\begin{proof}
The first statement is obvious. To prove (i), let  $x$ be any other solution of \eqref{E1}. 
Consider a Lyapunov function defined by
 $V(t)=\sum_{i=1}^2 |\ln x_i-\ln \bar x_i|+ x_3,  t \geq 0.$
Calculating of the right derivative $D^{+}V(t)$ of  $V(t)$ along the solutions of \eqref{E1} gives
%%%%%%%%%%%%%%%%%%%%%%%%%%%%% E28
\begin{align} \label{E28}
&D^+V(t) \notag\\
= &\sum_{i=1}^2 \mathop{\rm sgn}(x_i-\bar x_i) (\frac{x_i'}{x_i}-\frac{\bar x_i'}{\bar x_i}) +x_3' \notag\\
= &\sum_{i\ne  j}^2 \Big\{\mathop{\rm sgn}(x_i-\bar x_i) \Big[(a_i-b_{ii}x_i-b_{ij}x_j\notag\\
&-\frac{c_i x_i x_3}{\alpha+\beta x_i+\gamma x_3})-(a_i-b_{ii} \bar x_i-b_{ij} \bar x_j)\Big] \Big\}\notag\\
&+(-a_3+\sum_{i=1}^2 \frac{d_i x_i}{\alpha+\beta x_i+\gamma x_3}) x_3\notag\\
=&\sum_{i\ne  j}^2 \left\{[ -b_{ii}|x_i-\bar x_i|-b_{ij} (x_j-\bar x_j) \mathop{\rm sgn}(x_i-\bar x_i)  ]\right\}\notag\\
&+ x_3 \Big [-a_3+\sum_{i=1}^2 \Big\{\frac{d_i x_i}{\alpha+\beta x_i+\gamma x_3}\notag\\
&\hspace{3cm}-\frac{c_i x_i \,\mathop{\rm sgn}(x_i-\bar x_i)}{\alpha+\beta x_i+\gamma x_3}\Big\}\Big].
\end{align}
Then
\begin{align}
D^+V(t) \leq & \sum_{i\ne  j}^2 (b_{ij}-b_{jj}) |x_j-\bar x_j| \notag\\
&+\frac{(c_1+c_2+d_1+d_2-\beta a_3) x_3}{\beta}. \label{E13}
\end{align}
By assumptions in (i) and the periodicity of parameters, there exist $\mu_1>0$ such that
\begin{align*}
\max_{t\in [0,\omega], 1\leq i\ne j\leq 2} &\left\{ \frac{c_1+c_2+d_1+d_2-\beta a_3}{\beta}, b_{ij}-b_{jj}\right\}\\
&<-\mu_1.
\end{align*}
Thus, by integrating both the hand sides of \eqref{E13} from  $0$ to $t$, we observe that\\
\centerline{$V(t)+\mu_1 \int_0^{t}\sum_{i=1}^3 |x_i-\bar x_i|ds \leq V(0) <\infty$}
for every $t \geq 0$.
Hence,
$ \sum_{i=1}^3 |x_i-\bar x_i| \in L^1([0, \infty)).$

On the other hand, by the periodicity,  $x_i$ and  $\bar x_i \, (i\geq 1)$ have bounded derivatives on $[0,\infty)$. As a consequence,
$ \sum_{i=1}^3 |x_i-\bar x_i|$ is uniformly continuous on $[0, \infty)$. 
Therefore, by using the Barbalat lemma \cite{TVT2}, we conclude that\\
 \centerline{$\underset{t \to \infty}{\lim}\sum_{i=1}^3 |x_i-\bar x_i|=0,$}
i.e.  $\bar x$ is globally asymptotically stable. 

Similarly,  we obtain the conclusions in (ii) and (iii) by using the following inequalities, respectively.
\begin{align*}
&D^+V(t)\\
 =&\sum_{i\ne  j}^2 [ -(b_{ii}+b_{ji}) |x_i-\bar x_i|]+ x_3 \Big [-a_3\notag\\
&+\sum_{i=1}^2 \Big\{\frac{d_i x_i}{\alpha+\beta x_i+\gamma x_3}-\frac{c_i x_i \,\mathop{\rm sgn}(x_i-\bar x_i)}{\alpha+\beta x_i+\gamma x_3}\Big\}\Big]\\
\leq & \sum_{i\ne  j}^2 [ -(b_{ii}+b_{ji}) |x_i-\bar x_i|]\notag\\
&+\frac{(c_1+c_2+d_1+d_2-\beta a_3) x_3}{\beta},
\end{align*}
and
\begin{align*}
&D^+V(t) \\
=&\sum_{i\ne  j}^2 [ -(b_{ii}+b_{ji}) |x_i-\bar x_i|]+ x_3 \Big [-a_3\notag\\
&+\sum_{i=1}^2 \Big\{\frac{d_i x_i}{\alpha+\beta x_i+\gamma x_3}-\frac{c_i x_i }{\alpha+\beta x_i+\gamma x_3}\Big\}\Big]\\
\leq & \sum_{i\ne  j}^2 \left[ -(b_{ii}+b_{ji}) |x_i-\bar x_i|\right]+\frac{(d_1+d_2-\beta a_3) x_3}{\beta}.
\end{align*}
We complete the proof.
\end{proof}
\subsection{Numerical examples}
In this subsection, we exhibit some numerical examples which show the convergence of  positive solutions of  \eqref{E1} to  periodic solutions of  \eqref{E1}. 
Set $a_1=3+\sin(8t); a_2=5.5-0.2\cos(8t); a_3=0.4-0.3\cos(8t); b_{11}=2+\cos(8t); b_{22}=5+0.4\sin(8t); b_{12}=0.04-0.02\sin(8t); b_{21}=0.15-0.1\cos(8t); c_1=0.5-0.4\sin(8t);$ $ c_2=0.4-0.3\sin(8t); \alpha=0.03-0.02\cos(8t);
\beta=0.3+0.2\cos(8t);
\gamma=2-\sin (8t);
d_1=3+2 \sin(8t);
 d_2=3-2 \sin(8t);
$ and an initial value $(x_1(0),x_2(0),x_3(0))=(0.5,0.7,1).$ Figure \ref{Fig1} shows the behavior of the solution of \eqref{E1}. It is seen that the solution converges to a positive periodic solution of \eqref{E1}.
 \begin{figure}[h] 
\begin{center}
\includegraphics[scale=0.7]{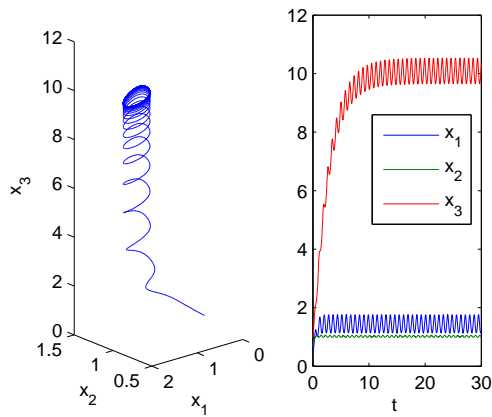}
\end{center}
\vspace*{9pt}
\caption{ A solution of  \eqref{E1} which converges to a positive periodic solution of \eqref{E1}}
\label{Fig1}
\end{figure}

We now set $a_3=4-0.3\cos(8t)$ and $\beta=3+0.2\cos(8t)$ and retain other parameters as above. Figure \ref{Fig2} gives the behavior of the positive solution of \eqref{E1}. It converges to the boundary periodic solution of \eqref{E1}.

 \begin{figure}[h] 
\begin{center}
\includegraphics[scale=0.7]{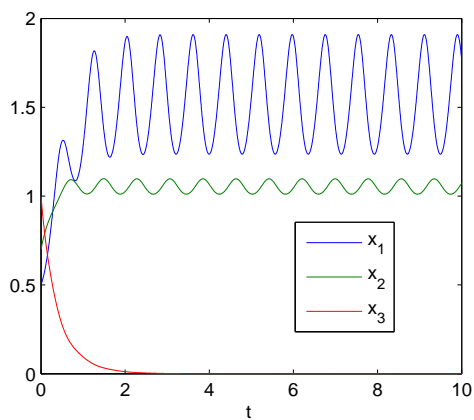}
\end{center}
\vspace*{9pt}
\caption{A  positive solution of  \eqref{E1} which converges to the boundary periodic solution of \eqref{E1}}
\label{Fig2}
\end{figure}
\section*{Acknowledgement}
 The work of the last author is supported by JSPS KAKENHI Grant Number 20140047. The authors would like to thank the anonymous referees for their helpful suggestions which improved the
paper. 
%%%%%%%%%%%%%%%%% BIBLIOGRAPHY IN THE LaTeX file !!!!! %%%%%%%%%%%%%%%%%%%%%%


\begin{thebibliography}{9}
\bibitem{APS} R. Arditi, N. Perrin, H. Saiah,  ``Functional response and heterogeneities: an experimental test with cladocerans",
\emph{OIKOS}, Vol. 60, pp. 69--75, 1991.
\bibitem{Be}	J. R. Beddington,   ``Mutual interference between parasites or predators and its effect on searching efficiency", \emph{J. Animal Ecol.}, Vol. 44, pp.  331--340, 1975.
%\bibitem{BS} D. D. Bainov, P. S. Simeonov, \emph{Impulsive Differential Equations: Periodic Solutions and Applications}, Pitman Monographs and Surveys in Pure and Applied Mathematics, 1993.
\bibitem{CC1}  R. S.  Cantrell, C. Cosner,  ``Effects of domain size on the persistence of populations in a diffusive food chain model with DeAngelis-Beddington functional response", \emph{Natural Resource Modelling}, Vol. 14, pp. 335--367, 2001.
\bibitem{CC2}   R. S.  Cantrell, C. Cosner,  ``On the dynamics of predator-prey models with the Bedding-DeAngelis functional response",  \emph{J. Math. Anal. Appl.}, Vol. 257, pp. 206--222, 2001.
%\bibitem{CC3}  R. S.  Cantrell, C. Cosner,  \emph{Spatial Ecology via Reaction-Diffusion Equations}, Wiley, Chichester, 2003.
\bibitem{CD}  C.  Cosner, D. L. DeAngelis, J. S.  Ault, D. B.  Olson, ``Effects of spatial grouping on the functional response of predators", \emph{Theoret. Population Biol.}, Vol. 56, pp. 65--75, 1999.
\bibitem{DGO} D. L. DeAngelis, R. A.  Goldstein, R. V.  O'Neill, ``A model for trophic interaction",  \emph{Ecology}, Vol. 56, pp. 881--892, 1975.
\bibitem{Do} P. M. Dolman, ``The intensity of interference varies with resource density: evidence from a field study with snow buntings",  \emph{Plectrophenax nivalis, Oecologia},  Vol. 102, pp. 511--514, 1995.
\bibitem{GM} R. E. Gaines, J. L. Mawhin, \emph{Coincidence Degree and Nonlinear Differential Equations}, Springer, Berlin, 1977.
\bibitem{Lb} B. Lisena, ``Global stability in periodic competitive systems",   \emph{Nonlinear Anal. Real World Appl.},  Vol. 5, pp. 613--627, 2004.
\bibitem{L3} Y. Li, ``Periodic solution of a periodic delay predator-prey system", \emph{Proc. Amer. Math. Soc.}, Vol. 127, pp. 1331--1335, 1999.
\bibitem{Linh}	 N. T. H. Linh, T. V. Ton, ``Dynamics of a stochastic ratio-dependent predator-prey model", \emph{Anal. Appl.} Vol. 9, pp. 329--344, 2011.
\bibitem{JE} C. Jost, S. Ellner, ``Testing for predator dependence in predator-prey dynamics: a nonparametric approach",  \emph{Proc. Roy. Soc. London Ser. B},  Vol. 267, pp. 1611--1620, 2000.
%\bibitem{KW} W. Krawcewicz, J. Wu,  \emph{ Theory of Degrees, with Applications to Bifurcations and Differential Equations}, John Wiley, New York, 1997.
%\bibitem{S} S. H. Saker, ``Oscillation and global attractivity of hematopoiesis model with delay time",  \emph{Appl. Math. Comp.}, Vol. 136, pp. 27--36, 2003.
\bibitem{SG} G. T. Skalski, J. F. Gilliam, ``Functional responses with predator interference: viable alternatives to the Holling type II model", \emph{Ecology}, Vol. 82, pp. 3083--3092, 2001.
\bibitem{TVT}  T. V. Ton, ``Dynamics of species in a non-autonomous Lotka-Volterra system",
 \emph{Acta Math. Acad. Paedagog. Nyhazi.}, Vol. 25, pp. 45--54, 2009. 
\bibitem{TVT1} 	T. V. Ton, A. Yagi, ``Dynamics of a stochastic predator-prey model with the Beddington-DeAngelis functional response", \emph{Commun. Stoch. Anal.} Vol. 5, pp. 371--386, 2011.
\bibitem{TVT2}	T. V. Ton, N. T. Hieu, ``Dynamics of species in a model with two predators and one prey", \emph{Nonlinear Anal.} Vol. 74, pp. 4868--4881, 2011.
\bibitem{TVT3} A. Yagi, T. V. Ton, ``Dynamic of a stochastic predator-prey population",  \emph{Appl. Math. Comput.} Vol. 218, pp. 3100--3109, 2011.
\end{thebibliography}
\end{document}